\documentclass[11pt,a4paper]{amsart}
\pdfoutput=1
\usepackage{amsmath,amsfonts,amssymb,amsthm,amscd,stmaryrd}
\usepackage{dsfont,mathrsfs,enumerate}
\usepackage{curves,epic,eepic}
\usepackage{appendix}

\usepackage[style=numeric-comp]{biblatex}
\usepackage[colorlinks=true,
linkcolor=red, citecolor=blue]{hyperref}
\usepackage{url}
\usepackage{cancel}
\usepackage{pdflscape}
\usepackage{float}
\usepackage{makecell}
\usepackage{nicefrac}

\usepackage[pdftex]{graphicx}
\usepackage{psfrag}
\usepackage{epic}
\usepackage{eepic}
\usepackage[matrix,arrow,curve]{xy}
\input{epsf}
\usepackage{tikz}
\usetikzlibrary{calc}
\usetikzlibrary{intersections}
\usetikzlibrary{through}
\usetikzlibrary{backgrounds}

\usepackage{algpseudocode}
\usepackage{algorithm}

\newtheorem{thm}{Theorem}[section]
\newtheorem{lemma}[thm]{Lemma}
\newtheorem{prop}[thm]{Proposition}

\newtheorem{conj}[thm]{Conjecture}
\theoremstyle{definition}

\newtheorem{rem}[thm]{Remark}

\DeclareMathOperator{\Z}{\mathds{Z}}
\DeclareMathOperator{\Q}{\mathds{Q}}
\DeclareMathOperator{\R}{\mathds{R}}

\DeclareMathOperator{\area}{\mathrm{area}}
\DeclareMathOperator{\Area}{\mathrm{Area}}
\DeclareMathOperator{\maxarea}{\mathrm{Maxarea}}
\DeclareMathOperator{\lw}{\mathrm{lw}}

\DeclareMathOperator{\ehr}{\mathrm{ehr}}

\DeclareMathOperator{\length}{\mathrm{length}}

\DeclareMathOperator{\width}{\mathrm{width}}

\DeclareMathOperator{\conv}{\mathrm{conv}}

\newcommand{\floor}[1]{\left\lfloor #1 \right\rfloor}

\def\<{\langle}
\def\>{\rangle}

\begin{document}
	
	\title[Rational polygons]{Generalizations of Scott's inequality
    and Pick's formula to rational polygons}
	
	\author[Martin Bohnert]{Martin Bohnert}
	\address{Mathematisches Institut, Universit\"at T\"ubingen,
		Auf der Morgenstelle 10, 72076 T\"ubingen, Germany}
	\email{martin.bohnert@uni-tuebingen.de}
	
	\author[Justus Springer]{Justus Springer}
	\address{Mathematisches Institut, Universit\"at T\"ubingen,
		Auf der Morgenstelle 10, 72076 T\"ubingen, Germany}
	\email{justus.springer@uni-tuebingen.de}
	
\begin{abstract}
    We prove a sharp upper bound on the number of boundary lattice
    points of a rational polygon in terms of its denominator and the
    number of interior lattice points, generalizing Scott's
    inequality. We then give sharp lower and upper bounds on the area
    in terms of the denominator, the number of interior lattice
    points, and the number of boundary lattice points, which can be seen
    as a generalization of Pick's formula. Minimizers and maximizers
    are described in detail. As an application, we derive bounds for
    the coefficients of Ehrhart quasipolymials of half-integral
    polygons.
\end{abstract}

	\maketitle
	
	\thispagestyle{empty}
	
	\section{Introduction}

    By a \emph{rational polygon} \( P \subseteq \R^2 \), we mean the
    convex hull of finitely many points in \( \Q^2 \). We call the
    smallest positive integer \( k \) such that \( kP \) has integral
    vertices the \emph{denominator} of \( P \). Scott \cite{Sco76}
    showed that for rational polygons of denominator one (henceforth
    called \emph{lattice polygons}), the numbers \( i \) and \( b \)
    of interior and boundary lattice points of \( P \) satisfy \( b
    \leq 9 \) for \( i = 1 \) and \( b \leq 2i+6 \) for \( i \geq 2 \)
    and this bound is sharp. In particular, \( (i,b) = (1,9) \) holds
    if and only if \( P \) is affine unimodular equivalent to the
    threefold standard lattice triangle. We prove the following
    generalization of Scott's inequality to rational polygons:

\begin{thm}
	\label{thm:scott_generalization}

	Let \( P \) be a rational polygon of denominator \( k \geq 2 \) with
	\( i \geq 1 \) interior and \( b \) boundary lattice points. Then \( b \leq
	(k+1)(i+1) + 3 \) and equality holds if and only if \( P \) is affine
	unimodular equivalent to the triangle \(
	\conv\left((0,\frac{1}{k}),(0,-1),((k+1)(i+1),-1)\right) \).

\end{thm}

Next, we provide area bounds. Recall that Pick's formula \cite{Pic99}
states that the area of a lattice polygon with \( i \) interior and
\( b \) boundary lattice points equals \( i+\frac{b}{2}-1 \). For
rational polygons of denominator at least two, fixing \( i \) and \(
b \) no longer uniquely determines the area. Instead, we provide
sharp lower and upper bounds, where we also completely
describe minimizers and maximizers.

\begin{thm}[See also figures~\ref{fig:area_minimizers_collinear_integer_hull}
    and~\ref{fig:area_minimizers}]
	\label{thm:area_lower_bound}

	Let \( P \) be a rational polygon of denominator \( k \geq 2 \) with
	\( i \geq 1 \) interior and \( b \) boundary lattice points. Set \(
	b_{\max} = (k+1)(i+1)+3 \). If \( Q := \conv(P \cap \Z^2) \) is
	two-dimensional and \( b \geq 2 \), we have
	\[
		\area(P)\ \geq\ \frac{i(k+1)+1}{2k} + \frac{b}{2} - 1.
	\]
	Here, equality holds if and only if \( P \) is equivalent to one of
	the following polygons:
	\[
		\begin{array}{lll}
			\mathrm{(0a)}                                     &
            \conv((0,-1),(i(k+1)-k+1,-1),(-\frac{1}{k},\frac{1}{k}))
			                                                & \qquad \text{for } b =
			b_{\max}-2(k+1)                                                                                            \\[5pt]
			\mathrm{(1a)} \qquad                              &
			\conv((0,-1),(b-2,-1),(i,0),(-\frac{1}{k},\frac{1}{k}))
			                                                & \qquad \text{for } 2 \leq b \leq
			b_{\max}-(k+1)                                                                                             \\[5pt]
			\mathrm{(1b)} \qquad                              &
			\conv((0,-2),(2,0),(-\frac{1}{k},\frac{1}{k})), &
			\qquad \text{for } (i,b) = (3,3)                                                                           \\[5pt]
			\mathrm{(2a)} \qquad                              &
			\conv((0,0),(0,-1),(b-3,-1),(i+1,0),(\frac{x}{k},\frac{1}{k}))
			                                                & \qquad \text{for } 3 \leq b \leq
			b_{\max}                                                                                                   \\[5pt]
			\mathrm{(2b)} \qquad                              &
			\conv((0,0),(0,-2),(2,0),(\frac{x'}{k},\frac{1}{k}))
			                                                & \qquad \text{for } (i,b) = (1,5),
		\end{array}
	\]
	where \( x = 0, \dots, \lfloor \frac{b_{\max}-b}{2} \rfloor \) and \(
	x' = 0, \dots, k \). Moreover, if \( \dim(Q) < 2 \), we have \( b
	\leq 2 \) and
	\[
		\area(P) \geq
		\begin{cases}
			\frac{i}{k} - \frac{1}{k} + \frac{3}{2k^2}, & b = 0  \\
			\frac{i}{k} + \frac{1}{2k^2},               & b = 1  \\
			\frac{i}{k} + \frac{1}{k},                  & b = 2. \\
		\end{cases}
	\]
	Here, equality holds if and only if \( P \) is equivalent to one of
	the following polygons:
	\[
		\begin{array}{lll}
			\mathrm{(0c)} & \conv((1,\frac{1}{k}),(1-\frac{1}{k},-\frac{1}{k}),(i+\frac{1}{k},0)),
			            & \text{ for } b = 0,                                                    \\[5pt]
			\mathrm{(1c)} & \conv((1,\frac{1}{k}),(1-\frac{1}{k},-\frac{1}{k}),(i+1,0)),
			            & \text{ for } b = 1,                                                    \\[5pt]
			\mathrm{(2c)} & \conv((0,0),(0,\frac{1}{k}),(\frac{x}{k},-\frac{1}{k}),(i+1,0)),
			            & \text{ for } b = 2,
		\end{array}
	\]
	where \( x = 0, \dots, k(i+1) \). All polygons listed above are
	pairwise non-equivalent.

\end{thm}

\begin{thm}[See also figure~\ref{fig:area_maximizers}]
	\label{thm:area_upper_bound}

	Let \( P \) be a rational polygon of denominator \( k \geq 4 \) with
	\( i \geq 1 \) interior and \( b \) boundary lattice points. Set \(
	b_{\max} := (k+1)(i+1)+3 \) and \( \tilde b := b_{\max}-b \). Then
	\[
		2k^2\area(P)\ \leq\ k(k+1)^2(i+1) -
		\begin{cases}
			\tilde b,              & b \in \{b_{\max}, b_{\max}-1\} \\
			(2 + k(\tilde b - 2)), & 1 \leq b \leq b_{\max}-2       \\
			(3 + k(\tilde b - 2)), & b = 0.
		\end{cases}
	\]
	Here, equality holds if and only if \( P \) is equivalent to one the
	following polygons: \setlength{\arraycolsep}{0pt}
	\[
		\begin{array}{lll}
			\mathrm{(0a)} \qquad                      & \conv\left((0,\frac{1}{k}),(\frac{1}{k},-1),((k+1)(i+1)-\frac{1}{k},-1)\right)
			                                          &
			\quad \text{for } b = b_{\max} - 4                                                                                                       \\[6pt]

			\mathrm{(0b)} \qquad \qquad               & \conv\left((0,\frac{1}{k}),(\frac{1}{k},-1),(1-\frac{1}{k},-1),
			(k(i+1)-\frac{1}{k},\frac{1}{k}-1)\right) & \quad \text{for } b =
			0                                                                                                                                        \\[6pt]
			\mathrm{(1a)} \qquad                      & \conv\left((0,\frac{1}{k}),(\frac{1}{k},-1),(b-\frac{1}{k},-1),(k(i+1),\frac{1}{k}-1)\right)
			                                          &
			\quad \text{for } 1 \leq b \leq b_{\max}-3,                                                                                              \\[6pt]
			\mathrm{(2a)} \qquad                      & \begin{array}{ll}
				                                            \conv\big( & (0,\frac{1}{k}),(0,\frac{1}{k}-1),
				                                            (x+\frac{1}{k},-1),                                             \\
				                                                       & (x+b-1-\frac{1}{k},-1),(k(i+1),\frac{1}{k}-1)\big)
			                                            \end{array}
			                                          &
			\quad \text{for } 2 \leq b \leq b_{\max}-2,
			\\[10pt]
			\mathrm{(2b)} \qquad
			                                          & \conv\left((0,\frac{1}{k}),(0,\frac{1}{k}-1),
			(\frac{1}{k},-1), ((k+1)(i+1),-1)\right)  &
			\quad \text{for } b =
            b_{\max} - 1, \\[6pt]
			\mathrm{(2c)}\qquad                       & \conv\left((0,\frac{1}{k}),(0,-1),((k+1)(i+1),-1)\right)
			                                          &
			\quad \text{for } b = b_{\max},                                                                                                          \\
		\end{array}
	\]
	where \( x = 0, \dots, \lfloor \frac{\tilde b}{2} \rfloor - 1 \). All
	polygons listed above are pairwise non-equivalent.

\end{thm}

After we prove
Theorems~\ref{thm:scott_generalization}-\ref{thm:area_upper_bound} in
section~\ref{sec:proofs_theorems_1.1_1.2_and_1.3}, we turn to the case
\( k = 2 \) in section~\ref{sec:half_integral_polygons}. We found that
in this case, the area bound from Theorem~\ref{thm:area_upper_bound}
can be violated precisely if \( i = 1 \) or \( b = 0 \). We obtain the
following sharp upper bound:

\begin{thm}
    \label{thm:half_integral_area_upper_bound}
	Let \( P \) be rational polygon of denominator \( 2 \) having \( i
    \geq 1 \) interior and \( b \) boundary lattice points. If \( i
    \geq 2 \), we have the sharp bound
    \[
        \area(P)\ \leq\ \frac{3}{2}i + \frac{1}{4}b + \frac{1}{8}
        \cdot
		\begin{cases}
			8, & 0 \leq b \leq 3i+4 \\
			7, & b = 3i+5           \\
			6, & b = 3i+6.
		\end{cases}
    \]
    Moreover, if \( i = 1 \), we have the sharp bound
    \[
        \area(P)\ \leq\ \frac{1}{4}b + \frac{1}{8}\cdot
		\begin{cases}
            21, & 0 \leq b \leq 6 \\ 
            20, & b = 7 \\ 
            19, & b = 8 \\ 
            18, & b = 9.
		\end{cases}
    \]
\end{thm}

\section{Proofs of Theorems 1.1, 1.2 and 1.3}
\label{sec:proofs_theorems_1.1_1.2_and_1.3}

For any polygon \( P \), we write \( i(P) \) for the number of
interior lattice points and \( b(P) \) for the number of boundary
lattice points. We call the lattice polygon \( \conv(P \cap \Z^2) \)
the \emph{integer hull} of \( P \). We call two polygons \( P \) and
\( Q \) \emph{equivalent}, if \( P = UQ + z \) holds for a unimodular
matrix \( U \in \Z^{2\times 2} \) and some \( z \in \Z^2\).

Let us recall the classification of lattice polygons without interior
lattice points.

\begin{prop}[See {\cite[Theorem 4.1.2]{Koe91}}]
    \label{prp:hollow_lattice_polygons}

    Let \( P \) be a lattice polygon with \( i(P) = 0 \). Then \( P \)
    is either equivalent to the twofold standard triangle \(
    \conv((0,0),(2,0),(0,-2)) \) or to a polygon of the form 
    \(\conv((0,0),(0,-1),(n,-1),(m,0)) \) with integers \( n \geq m \geq 0 \).

\end{prop}

\begin{proof}[Proof of Theorem~\ref{thm:scott_generalization}.]
    Consider the integer hull
	\( Q := \conv(P \cap \Z^2) \). We may assume \( Q \) to be
	two-dimensional, otherwise we have \( b \leq 2 \)
	and the claim holds trivially. If \(i(Q)>0\), then we
	get with Scott's inequality
	\[
		b\ \leq\ b(Q)\ \leq\ 2i(Q)+7\ \leq\ 2i+7\ \leq\ (k+1)(i+1)+3.
	\]
    This can only be an equality for \( (k,i) = (2,1) \) and \( Q \)
    being the threefold standard lattice triangle. But since \( P \)
    is not a lattice polygon, it must contain a non-integral vertex
    outside of \( Q \). Since the threefold standard lattice triangle
    has a lattice point in the relative interior of each edge, this
    implies \( b < b(Q) \), hence the inequality is strict after all.

    If \( i(Q) = 0 \), Proposition~\ref{prp:hollow_lattice_polygons}
    says that \( Q \) is either the twofold standard
	lattice triangle or of the form \( \conv((0,0),(0,-1),(n,-1),(m,0))
	\) for integers \( n \geq m \geq 0 \). In the former case, we have
	\(b \leq b(Q) = 6 \) and the claim holds trivially. Let us treat the
	latter case. Since \( P \) has at least one interior lattice point,
	there must be a vertex \( v = (x,y) \) of \( P \) which lies outside
	of the strip \( \R \times [-1,0] \). If \( y < -1 \), we have \( i
	\geq n-1 \), hence
	\[
		b\ \leq\ m+3\ \leq\ n+3\ \leq\ i+4\ <\ (k+1)(i+1)+3.
	\]
	Now assume \( y > 0 \). Then we maximize the number of boundary
	lattice points if \( (-1,l) \) are all boundary points of \( P \) for
	\( 0 \leq l \leq n \) and \( n \) is as large as possible. This
	happens when \( y \) is as small as possible i.e. \( y = \frac{1}{k}
	\). By convexity, the largest possible \( n \) is \( 
    (k+1)(i+1) \). Then we obtain the triangle from the assertion and
    \( b = n + 3 \) as claimed. 
\end{proof}

\begin{rem}
    \label{rem:generalized_scott_k1}
    Plugging \( k = 1 \) into
    Theorem~\ref{thm:scott_generalization} gives \( i \leq 2i+5 \),
    which is less than Scott's inequality. The difference is explained
    as follows: First, the threefold standard lattice triangle is an
    exception which only exists for \( k = 1 \). Second, if we take \(
    k = 1\) in the triangle from
    Theorem~\ref{thm:scott_generalization}, the vertex \(
    (0,\frac{1}{k}) \) is itself a lattice point, hence it has one
    more boundary lattice point than expected.
\end{rem}

In the following proof, it will be convenient for us to normalize our
area measure with respect to the \( \frac{1}{k} \)-fold standard
triangle instead of the unit square. To that end, we define the
\emph{\( k \)-normalized area} of \( P \) as 
\[ 
    \Area_k(P)\ :=\ 2k^2\area(P), 
\]
where \( \area(P) \) denotes the standard euclidian
area. Note that \( \Area_k(P) \) is always an integer.

\begin{proof}[Proof of Theorem~\ref{thm:area_lower_bound}.]
We treat first the case \( \dim(Q) = 2 \). In terms of the \( k
\)-normalized area, we have to show the following:
\[
\Area_{k}(P)\ \geq\ k^2(i+b-2)+k(i+1).
\]
Pick's formula gives us a first bound
\[
	\Area_{k}(P)\ \geq\ \Area_{k}(Q)\ =\ k^2\left(2i(Q)+b(Q)-2\right).
\]
Consider an interior lattice point \( A \) of \( P \) which lies on
the boundary of \( Q \) and let \( B \) be a boundary lattice point of
\( Q \) adjacent to \( A \). Pick a rational point \( C \in (P
\setminus Q) \cap \frac{1}{k}\Z^2 \) lying over the edge of \( Q \)
which contains \( A \) and \( B \) and consider the triangle with
vertices \( A, B \) and \( C \). Its base has length one and its height
is at least \( \frac{1}{k} \). Hence it has \( k \)-normalized area of at
least \( k \). Choosing the same \( C \) for as many pairs \( (A,B) \)
as possible gives us at least \( i - i(Q) + 1 \) triangles of this
form with disjoint interiors. We obtain \(\Area_{k}(P) - \Area_{k}(Q)\
\geq\ k(i - i(Q) + 1) \).

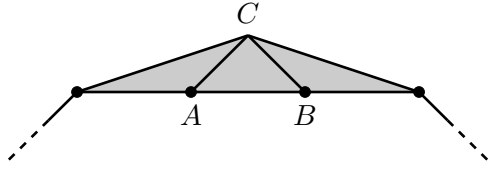
\begin{center}
    \begin{tikzpicture}[x=1.5cm,y=1.5cm]

        \draw [fill=black] (0,0) circle (2.0pt);
        \draw [fill=black] (1,0) circle (2.0pt);
        \draw [fill=black] (2,0) circle (2.0pt);
        \draw [fill=black] (3,0) circle (2.0pt);

        \draw[line width = 1pt] (-0.25,-0.25) -- (0,0) -- (3,0) --
            (3.25,-0.25);
        \draw[line width = 1pt, dashed] (-0.25,-0.25) -- (-0.6,-0.6);
        \draw[line width = 1pt, dashed] (3.25,-0.25) -- (3.6,-0.6);

        \draw[line width=1pt] (0,0) -- (1.5,0.5);
        \draw[line width=1pt] (1,0) -- (1.5,0.5);
        \draw[line width=1pt] (2,0) -- (1.5,0.5);
        \draw[line width=1pt] (3,0) -- (1.5,0.5);

        \fill[opacity=0.2] (0,0) -- (3,0) -- (1.5,0.5) -- cycle;

        \node at (1,-0.2) {\( A \)};
        \node at (2,-0.2) {\( B \)};
        \node at (1.5,0.7) {\( C \)};

    \end{tikzpicture}
\end{center}
Using \( b(Q)=b+i-i(Q) \), we compute
\begin{align*}
    \Area_k(P) & \geq k^2(2i(Q) + b(Q) - 2) + k(i-i(Q)+1) \\
              & = k^2(i+b-2) + k(i+1) + i(Q)(k^2-k) \\
              & \geq k^2(i+b-2) + k(i+1).
\end{align*}
To get equality, we need \( i(Q) = 0 \) and \( \Area_k(P) - \Area_k(Q) =
k(i+1) \). The latter happens if and only if there is a unique point
\( C \in (P \setminus Q) \cap \frac{1}{k}\Z^2 \) having lattice
distance \( \frac{1}{k} \) to all relevant edges of \( Q \). By
Proposition~\ref{prp:hollow_lattice_polygons}, \( Q \) is either the
twofold standard triangle or of the form \(
\conv((0,0),(0,-1),(n,-1),(m,0)) \) for \( n \geq m \geq 0 \). In the
former case, we get the polygons (1b) and (2b) from the assertion. In
the latter case, we get the polygons (0a), (1a) and (2a). Here, the
number before the letter denotes the number of boundary lattice points
of \( P \) on \( \R \times \{0\} \), which is at most two. Note that
by symmetry, for the polygons (2a) it is enough to consider \( x = 0,
\dots, \lfloor \frac{b_{\max} - b}{2} \rfloor \). Similarly, for (2b),
\( x' = 0, \dots, k \) is enough. See also
figure~\ref{fig:area_minimizers} for an illustration.

Now we consider \( \dim(Q) < 2 \), i.e. the lattice points of \( P \)
are collinear. In particular, we have \( b \leq 2 \). After an affine
unimodular transformation, we may assume \( Q \subseteq \R \times
\{0\} \). In terms of the \( k \)-normalized area, we have to show
\[
	\Area_k(P) \geq
	\begin{cases}
		2k(i+1),   & b = 2, \\
		2ki+1,     & b = 1, \\
		2k(i-1)+3, & b = 0.
	\end{cases}
\]
Note that by Pick's formula, we have
\begin{equation}
    \label{eq:pick_scaled}
    \Area_k(P)\ =\ \Area_1(kP)\ =\ 2i(kP) + b(kP) - 2.
\end{equation}

We start with the case \( b = 2 \). Then \( kQ \) has exactly \(
k(i+1)-1 \) lattice points in its relative interior, hence \( i(kP)
\geq k(i+1)-1 \). Moreover, \( kP \) must have at least one lattice
point (a vertex) above and below of \( Q \). Together with the two
boundary lattice points on \( Q \) itself, this gives \( b(kP) \geq 4
\). Plugging this into equation~(\ref{eq:pick_scaled}), we arrive at
the desired bound and equality holds if and only if \( i(kP) =
k(i+1)-1 \) and \( b(kP) = 4 \). Up to affine unimodular equivalence,
there are exactly \( k(i+1)+1 \) ways to choose boundary points of \(
kP \) without getting additional interior lattice points, which are
exactly the polygons (2c) from the assertion. 

For \( b = 1 \), we obtain with analogous reasoning \( i(kP) \geq ki
\) and \( b(kP) \geq 3 \). The bound again follows from
Equation~(\ref{eq:pick_scaled}). To get equality, there is up to
affine unimodular transformation only one possible choice of boundary
points of \( kP \), which gives us the polygon (1c) from the
assertion.

For \( b = 0 \), we have \( i(kP) \geq k(i-1)+1 \) and \( b(kP) \geq
3 \), which yields the desired bound. Again, to get equality, there
is only one possibility, namely the polygon (0c) from the assertion.
See also figure~\ref{fig:area_minimizers_collinear_integer_hull} for
an illustration.
\end{proof}

\begin{rem}
    \label{rem:area_minimizers_b_2}

    In Theorem~\ref{thm:area_lower_bound}, the polygon (1a) for \( b =
    2 \) plays a special role: While it is always an area minimizer
    with respect to polygons with two boundary lattice points having
    a two-dimensional integer hull, it is not in general an area
    minimizer with respect to all polygons with two boundary lattice
    points. This is because the bound for \( \dim(Q) < 2 \)
    and \( b = 2 \) (namely \( \area(P) \geq \frac{i+1}{k} \)) is in general smaller. The only exception is the
    case \( (k,i) = (2,1) \), where both bounds agree. This is the
    only case where (1a) for \( b = 2 \) truly is an area minimizer.
    
\end{rem}

We turn to upper bounds on the area, i.e.
Theorem~\ref{thm:area_upper_bound}. Let \( k \geq 2 \) and \( i \geq 1
\). Set \( b_{\max} := (k+1)(i+1)+3 \) and let \( 0 \leq b \leq
b_{\max} \). Then we intend to prove that the \( k \)-normalized
area of a rational polygon with denominator \( k \) having \( i \)
interior and \( b \) boundary lattice points is not greater than
\[
	\maxarea(k,i,b)\ :=\ k(k+1)^2(i+1) -
		\begin{cases}
			\tilde b,              & b \in \{b_{\max}, b_{\max}-1\} \\
			(2 + k(\tilde b - 2)), & 1 \leq b \leq b_{\max}-2       \\
			(3 + k(\tilde b - 2)), & b = 0
		\end{cases}
\]
where \( \tilde b := b_{\max} - b \).

\begin{lemma}
	\label{lem:area_bound_not_m1p1}

	Let \( P \) be a rational polygon with denominator \( k \) having \(
	i \) interior and \( b \) boundary lattice points. Assume \( P \) is
	not equivalent to a polygon in \( \R \times [-1,1] \). Then we have

	\[
		\Area_{k}(P) \leq
		\max\left(k^2(4i+5),\ \frac{1}{2}k(k+2)^2(i+1)\right).
	\]
	In particular, for \( k \geq 4 \) we have \( \Area_k(P) <
	\maxarea(k,i,b) \).
\end{lemma}

\begin{proof}
    See \cite[Proposition 7.2]{BS24} for the first bound on \(
    \Area_k(P) \). For the supplement, note that \( k \geq 4 \) implies
    \[
		\max\left(k^2(4i+5),\ \frac{1}{2}k(k+2)^2(i+1)\right)\ <\ 
        k^2(k+1)(i+1) - k - 3.
    \]
    and \( k^2(k+1)(i+1)-k-3 = \maxarea(k,i,0) \leq \maxarea(k,i,b) \).
\end{proof}

\begin{lemma}
    \label{lem:area_bound_not_small_schlauch}
    Let \( P \) be a rational polygon of denominator \( k \geq 4 \)
    having \( i \) interior and \( b \) boundary lattice points.
    Assume \( P \) is not equivalent to a polygon in \( \R \times
    [-1,\frac{1}{k}] \). Then \( \Area_k(P) < \maxarea(k,i,b) \).
\end{lemma}

\begin{proof}
By Lemma~\ref{lem:area_bound_not_m1p1}, we can assume that \( P
\subseteq \R \times [-1,1] \). Let \( h \in \{2,\dots, k\} \) be
minimal such that \( P \subseteq \R \times [-1,\frac{h}{k}] \). Then
by \cite[Proposition 3.1]{Boh23}, we have \(\Area_{k}(P)\leq
k\left(\frac{k^2}{h}+2k+h\right)(i+1).\) For \( k \geq 4 \), this
attains its maximum at \( h = 2 \). We obtain
\[
	\Area_k(P) \leq k\left( \frac{k^2}{2}+2k+2\right)(i+1) <
	k^2(k+1)(i+1)-k-3 \leq \maxarea(k,i,b),
\]
where we again used \( k \geq 4 \) to get the strict inequality in the
middle.
\end{proof}

\begin{proof}[Proof of Theorem~\ref{thm:area_upper_bound}.]
By Lemma~\ref{lem:area_bound_not_small_schlauch}, we may assume \( P
\subseteq \R \times [-1,\frac{1}{k}] \). Throughout our proof, we use
the following shorthand notation: For any \( -1 \leq y \leq
\frac{1}{k} \) as well as \( -1 \leq y_1 < y_2 \leq \frac{1}{k} \), we
write
\[
P_y\ :=\ P \cap (\R \times \{y\}), \qquad P_{[y_1,y_2]}\ :=\ P \cap
(\R \times [y_1,y_2]).
\]
After a translation, we can assume \( P_0 \subseteq [0,i+1] \times
\{0\} \). Next, we apply a shearing to achieve that the leftmost
vertex in \( P_{1/k} \) is \( (0,\frac{1}{k}) \). Convexity then
implies that \( P_{[-1,0]} \) is contained in the convex hull of \(
(0,0), (0,-1), ((k+1)(i+1),-1) \) and \( (i+1,0) \). In particular,
the line segment \( P_{-1} \) is contained in \( [0,(k+1)(i+1)] \times
\{-1\} \). Write \( b_0 \) and \( b_{-1} \) for the number of boundary
point of \( P \) in \( P_0 \) and \( P_1 \) respectively. Then \( b =
b_0 + b_{-1} \) and \( b_0 \in \{0,1,2\} \). The length of \( P_{-1}
\) is restricted by the number of lattice points \( b_{-1} \) on
it, i.e.
\begin{equation}
    \label{eq:length_P_minus_one}
	\length(P_{-1})\ \leq\ 
	\begin{cases}
        (k+1)(i+1), & b_{-1} = b_{\max}-2 \\
        (k+1)(i+1)-\frac{1}{k}, & b_{-1} = b_{\max}-3 \\
		b_{-1}+1-\frac{2}{k}, & 0 \leq b_{-1} \leq b_{\max}-4.
	\end{cases}
\end{equation}
Consider now the line segment \( P_{-1+1/k} \), which is contained in
\( [0,k(i+1)] \times \{-1+1/k\} \). We obtain the following bound:
\begin{equation}
    \label{eq:length_P_minus_one_plus_one_over_k}
    \length(P_{-1+1/k})\ \leq\ k(i+1) - \frac{2-b_0}{k+1}.
\end{equation}
We consider the decomposition \( P = P_{[-1,-1+1/k]} \cup
P_{[-1+1/k,1/k]} \). Note that the area of \( P_{[-1+1/k,1/k]} \) is
maximized if and only if (\ref{eq:length_P_minus_one_plus_one_over_k})
is an equality and \( (0,\frac{1}{k}) \) is the only vertex in \(
P_{1/k} \). This is because \( P_{[-1+1/k,0]} \) cannot be larger
anyway and using an additional vertex in \( P_{1/k} \) would lose more
area in \( P_{[-1+1/k,0]} \) than we could gain in \( P_{[0,1/k]} \).
On the other hand, \( P_{[-1,-1+1/k]} \) is just a trapezoid, hence
its area is maximized if and only if (\ref{eq:length_P_minus_one}) and
(\ref{eq:length_P_minus_one_plus_one_over_k}) are both equalities. We
now proceed by case distinction on \( b_0 \in \{0,1,2\} \) and
describe all polygons satisfying these conditions, which gives us all
area maximizers.

We first treat \( b_0 = 2 \), which implies \( b \geq 2 \) and \(
b_{-1} = b - 2 \). Here, we achieve equality in
(\ref{eq:length_P_minus_one_plus_one_over_k}) if we have edges
connecting the vertex \( (0,\frac{1}{k}) \) to \( (0,-1+\frac{1}{k})
\) and \( (k(i+1),-1+\frac{1}{k}) \). For \( b = b_{\max} \), there is
only one possibility to achieve equality in
(\ref{eq:length_P_minus_one}), namely the polygon (2c). For \( b =
b_{\max}-1 \), there are a priori two possiblieties of maximizing the
length of \( P_{-1} \). However, they are equivalent by symmetry,
hence there is only one polygon (2b). For \( 2 \leq b \leq b_{\max}-2
\), we obtain the family of polygons (2a), which arises from different
translations of the line segment \( P_{-1} \). A priori, we have \( x
= 0, \dots, \tilde b - 2 \). However, by symmetry, it is enough to
consider \( x = 0, \dots, \lfloor \frac{\tilde b}{2} \rfloor - 1 \).
Computing the area of (2a), (2b) and (2c), we arrive at the bound from
the assertion.

Next, we treat \( b_0 = 1 \). By symmetry, we can assume that the
unique boundary lattice point on \( P_0 \) is \( (i+1,0) \). Then we
achieve equality in (\ref{eq:length_P_minus_one_plus_one_over_k}) if
we have edges connecting the vertex \( (0,\frac{1}{k}) \) to \(
(\frac{1}{k},-1) \) and \( (k(i+1),-1+\frac{1}{k}) \). By convexity,
we can maximize \( \length(P_{-1}) \) (i.e. achieve equality in
(\ref{eq:length_P_minus_one})) only for \( b \leq b_{\max}-3 \),
which gives us the polygons (1a) from the assertion. Note that these
have exactly the same area as the polygons (2a) for fixed \( b \leq
b_{\max}-3 \) and we cannot reach this bound for \( b \) larger and \(
\length(P_{-1}) \) not maximal. Moreover, since \( (0,\frac{1}{k}-1)
\) is no longer a vertex for (1a), there is no flexibility in moving
\( P_{-1} \) around. Hence for each \( 1 \leq b \leq b_{\max}-3 \),
we only get one polygon of type (1a).

Lastly, for \( b_0 = 0 \), simultaneous equality in
(\ref{eq:length_P_minus_one}) and
(\ref{eq:length_P_minus_one_plus_one_over_k}) can only be reached for
\( b = b_{\max}-4 \), in which case connecting \( (0,\frac{1}{k}) \)
to \( (\frac{1}{k},-1) \) and \( (b+1-\frac{1}{k},-1) \) gives the
triangle (0a). It has the same area as the polygons corresponding
polygons (1a) and (2a). This leaves open the case \( b = 0 \). Here,
we cannot reach equality in
(\ref{eq:length_P_minus_one_plus_one_over_k}), since this would fix \(
P_{-1} \) and lead to \( b > 0 \). The best we can do in this case is
to connect \( (0,\frac{1}{k}) \) to \( (\frac{1}{k},-1) \) and \(
(k(i+1)-\frac{1}{k},\frac{1}{k}-1) \). The line segment \( P_{-1} \)
is then already uniquely determined, hence we arrive at the polygon
(0b). It has the claimed area from the assertion, which is notably
smaller than in the other cases by one unit of \( k \)-normalized area.
\end{proof}

\begin{rem}
    \label{rem:area_bounds_k_1}
    We can compare our area bounds to Pick's formula by plugging \(
    k = 1\) into Theorems~\ref{thm:area_lower_bound}
    and~\ref{thm:area_upper_bound}. Ignoring the cases \( b \leq 2 \)
    (which cannot happen for lattice polygons), we arrive at \(
    i+\frac{b}{2}-\frac{1}{2} \) for both the lower and upper bound.
    This is notably \( \frac{1}{2} \) more than what Pick's formula
    says. The difference can be explained as follows: All of our area
    maximizers and minimizers described in
    Theorems~\ref{thm:area_lower_bound} and~\ref{thm:area_upper_bound}
    have a rational vertex at \( y = \frac{1}{k} \). For \( k = 1 \),
    this would become a boundary lattice point, hence the polygon has
    one more boundary lattice point than it would have for \( k > 1
    \). Substituting \( b \mapsto b+1 \) into Pick's formula, we
    recover the same expression \( i + \frac{b}{2} - \frac{1}{2} \).
\end{rem}

\begin{rem}
\label{rem:area_maximizers_in_small_strip}

In the proof of Theorem~\ref{thm:area_upper_bound}, the condition \( k
\geq 4 \) was only neccessary to ensure \( P \subseteq \R \times
[-1,\frac{1}{k}] \) by applying the bounds from
lemmas~\ref{lem:area_bound_not_m1p1}
and~\ref{lem:area_bound_not_small_schlauch}. In particular, the area
bound of Theorem~\ref{thm:area_upper_bound} holds for all polygons in
\( \R \times [-1,\frac{1}{k}] \), even if \( k \in \{2,3\} \).
However, as we
will see in Section~\ref{sec:half_integral_polygons}, not all
area maximizers for \( k = 2 \) are contained in \( \R \times
[-1,\frac{1}{2}] \). For \( k = 3 \), we used our classification of
rational polygons \cite{BS24} to verify that
Theorem~\ref{thm:area_upper_bound} holds completely for \( 1 \leq i
\leq 5 \). In particular, there are no maximizers outside of \( \R
\times [-1,\frac{1}{3}] \). This leads us to conjecture that the
bounds from lemmas~\ref{lem:area_bound_not_m1p1}
and~\ref{lem:area_bound_not_small_schlauch} can be improved to work
also for \( k = 3 \).

\end{rem}

\begin{rem}
	\label{rem:area_attainers_in_between_min_max}
	Let \( P \) be an area maximizer of type (2a) as described in
	Theorem~\ref{thm:area_upper_bound} with at least three boundary
    lattice points. Taking \( Q := \conv(P \cap \Z^2)
	\), the polygon \( P' := \conv(Q \cup \{(0,\nicefrac{1}{k})\}) \)
	is then an area minimizer with the same number of interior and
	boundary lattice points as \( P \). Moreover, we can find a polygon
	that attains any area between \( \Area_k(P) \) and \( \Area_k(P') \)
	by means of the following three operations:
	\begin{enumerate}
		\item Shorten or lengthen the line segment \( P \cap \R \times \{-1\} \) by \(
		      \nicefrac{l}{k} \) on either side for \( 1 \leq l \leq k
              - 1 \). This lowers or increases \( \Area_k(P) \) by \( l
              \).
		\item Add or remove the vertex \( (\nicefrac{l}{k},\nicefrac{1}{k}) \) for \( 1 \leq
		      l \leq k-1 \). This increases or lowers \( \Area_k(P) \) by \( l \).
		\item Move in the vertices on \( P \cap \R \times \{-1+\nicefrac{1}{k}\} \)
		      by \( \nicefrac{1}{k} \). This lowers \( \Area_k(P) \) by \( k \).
	\end{enumerate}
	See figure~\ref{fig:area_attainers_in_between_min_max} for an
	example of how these operations can be used to attain any area
	between the lower and upper bounds.
\end{rem}

	\section{Half-integral polygons}
    \label{sec:half_integral_polygons}

    In this section, we prove the sharp upper bound for the area of
    half-integral polygons given in
    Theorem~\ref{thm:half_integral_area_upper_bound}. Note that it
    differs from the bound given in Theorem~\ref{thm:area_upper_bound}
    exactly in the cases \( i = 1 \) and \( b = 0 \), where it exceeds
    the general bound by \( \frac{1}{8} \) (\( = \) one unit of \( 2
    \)-normalized area). For \( i = 1 \), the half-integral polygons
    that violate the bound from Theorem~\ref{thm:area_upper_bound} are
    all subpolygons of the threefold standard lattice triangle, see
    figure~\ref{fig:half_integral_area_maximizers_two_interior_integral_lines}.
    For \( b = 0 \), the polygons violating
    Theorem~\ref{thm:area_upper_bound} are all realized in \( \R
    \times [-1,1] \). This class of half-integral area maximizers also
    occurs for \( b > 0 \), where their area agrees with the
    generic bound. The following lemma completely describes them.

 	\begin{lemma}[See also
        figure~\ref{fig:half_integral_area_maximizers_m1p1}]
        \label{lem:half_integral_area_bound_m1p1}
        Let \( P \subseteq \R \times [-1,1] \) be a rational polygon
        of denominator two with \( i \) interior and \( b \)
        boundary lattice points which cannot be realized in \( \R
        \times [-1,\frac{1}{2}] \). Then we have
        \[
            b \leq 2i+6, \qquad \Area_2(P) \leq 12i+2b+8.
        \]
        Moreover, for \( -1 \leq y \leq 1 \), write \( \ell_y \) for
        the length of the line segment \( P \cap (\R \times \{y\}) \)
        and \( b_y \) for the number of boundary lattice points on them. Then
        \( \Area_2(P) = 12i+2b+8 \) if and only if the following
        hold:
        \begin{itemize}
            \item \( (\ell_1,\ell_{-1}) = (b_1,b_{-1}) \) and \( \ell_0
                = i + \frac{2}{3} + \frac{1}{6}b_0 \),
            \item \( P \cap (\R \times [-\frac{1}{2},\frac{1}{2}]) \) is
                a trapezoid, i.e. \( P \) has no vertex on \( \R
                \times \{0\} \).
        \end{itemize}
	\end{lemma}
	\begin{proof}

        We clearly have \( b_1 \leq \ell_1 + 1 \) and \( b_{-1} \leq
        \ell_{-1} + 1 \) as well as \( b_0 \leq 2 \). Moreover, \(
        \ell_0 \leq i + 1 \) holds and by convexity, we have \(
        \ell_{-1} + \ell_1 \leq 2\ell_0 \). We obtain
        \[
            b\ =\ b_{-1} + b_0 + b_1\ \leq\ \ell_{-1} + \ell_1 + 4
            \ \leq\ 2\ell_0 + 4\ \leq\ 2i+6.
        \]
        To estimate the area,
        consider the trapezoids \( P_{-1} := P \cap (\R \times
        [-1,-\frac{1}{2}]) \) and \(
        P_1 := P \cap (\R \times [\frac{1}{2},1]) \). Note that these
        are non-trivial, since we assume \( P \) cannot be realized in
        \( \R \times [-1,\frac{1}{2}] \). Their area is given by
        \[
            \Area_2(P_{-1})\ =\ 2(\ell_{-1} + \ell_{-\frac{1}{2}}),
            \qquad \Area_2(P_1)\ =\ 2(\ell_{\frac{1}{2}} + \ell_1).
        \]
        Moreover, the area of \( P_0 := P \cap (\R \times
        [-\frac{1}{2},\frac{1}{2}]) \) is
        \[
        \Area_2(P_0)\ =\ 2(\ell_{-\frac{1}{2}} + \ell_0) + 2(\ell_0 +
        \ell_{\frac{1}{2}})\ =\ 4\ell_0 + 2(\ell_{-\frac{1}{2}} +
        \ell_{\frac{1}{2}}).
        \]
        Since \( P \) has denominator two, we have \(
        \ell_{-1} \leq b_{-1} \) and \( \ell_1 \leq b_1 \). Moreover,
        convexity forces \( \ell_{-\frac{1}{2}} + \ell_{\frac{1}{2}}
        \leq 2\ell_0 \). In total, we obtain
        \[
            \Area_2(P) = \Area_2(P_{-1}) + \Area_2(P_0) + \Area_2(P_1)
            \leq 12\ell_0 + 2(b_{-1} + b_1).
        \]
        Consider now the line segment \( [s,s+\ell_0] \times \{0\} :=
        P \cap (\R \times \{0\}) \). Since \( P \subseteq \R \times
        [-1,1] \) and the denominator of \( P \) is two, the possible
        values for the rational part of \( s \) are \( \{0,
        \frac{1}{6}, \frac{1}{4}, \frac{1}{3}, \frac{1}{2},
        \frac{2}{3}, \frac{3}{4}, \frac{5}{6} \} \). This implies the
        bound \( \ell_0 \leq i + \frac{2}{3} + \frac{1}{6}b_0 \). Hence we
        obtain
        \[
            \Area_2(P) \leq 12i+8+2(b_0+b_{-1}+b_1) = 12i+2b+8.
        \]
        Moreover, equality holds if and only if the conditions from
        the assertion hold. Note that \( P_0 \) being a trapezoid is
        equivalent to \( \ell_{-\frac{1}{2}} + \ell_{\frac{1}{2}} =
        2\ell_0 \). See also
        figure~\ref{fig:half_integral_area_maximizers_m1p1} for an
        illustration of some area maximizers of the described form.
	\end{proof}

    In our proof of Theorem~\ref{thm:half_integral_area_upper_bound},
    we use the concept of lattice width, which we now quickly recall.
    For a polygon \( P \subseteq \R^2 \) and a non-zero vector \( w
    \in \R^2 \), we define the \emph{width} of \( P \) in direction \(
    w \) as
    \[
        \width_w(P)\ :=\ \max_{v \in P} \langle w,v \rangle - \min_{v
        \in P} \langle w,v \rangle.
    \]
    The \emph{lattice width} \( \lw(P) \) of \( P \) is defined to be the minimum
    over \( \width_w(P) \) where \( w \in \Z^2 \setminus \{0\} \).
    By a \emph{lattice width direction}, we mean a \( w \in \Z^2
    \setminus \{0\} \) such that \( \lw(P) = \width_w(P) \).

\begin{proof}[Proof of
    Theorem~\ref{thm:half_integral_area_upper_bound}]

    After an affine unimodular transformation, we can
    assume that \( (0,1) \) is a lattice width direction of \( P \).
    Let \( n \) be the number of interior integral horizontal lines,
    i.e.
	\[
		n\ :=\ |\{j \in \Z \,\mid\,P^{\circ} \cap (\R \times \{j\})
		\neq \emptyset \}|.
	\]
    We first treat \( n = 1 \), where we distinguish two subcases:
    Either \( P \) can be realized in \( \R \times [-1,\frac{1}{2}] \)
    or it can be realized in \( \R \times [-1,1] \) but not in \( \R
    \times [-1,\frac{1}{2}] \). In the former case, the assertion
    follows from the proof of Theorem~\ref{thm:area_upper_bound}, see
    also Remark~\ref{rem:area_maximizers_in_small_strip}. In
    particular, we get sharpness for our bound if \( b \geq 1 \). In
    the latter case, we refer to
    lemma~\ref{lem:half_integral_area_bound_m1p1}, where we also
    obtain sharpness for \( b = 0 \).

    Now we consider \( n \geq 2 \). We may assume \( P \subseteq \R \times
	[0,n+1] \). For \( 0 \leq y \leq n+1 \), let \( \ell_y \) be the
	length of the line segment \( P \cap (\R \times \{y\}) \). Consider
	the numbers \( i_j := |P^{\circ}\cap(\Z \times \{j\})| \) for \( j =
	1, \dots, n \) as well as \( b_j := |\partial P \cap (\Z \times \{j\})| \) for \( j = 0, \dots, n+1 \). Note
	that we have
	\[
		\ell_0 \leq b_0, \qquad \ell_{n+1} \leq b_{n+1}, \qquad \ell_j
		\leq i_j+1\ \text{for}\ j = 1, \dots, n.
	\]
	Consider the subpolygons \( P_j := P \cap (\R \times [j-\frac{1}{2},
		j+\frac{1}{2}]) \) for \( j = 0, \dots, n+1 \). For \( j = 1, \dots,
	n \), we can estimate their area by
	\[
		\Area_2(P_j) = 4\ell_j + 2(\ell_{j-\frac{1}{2}} + \ell_{j+\frac{1}{2}})
		\leq 8\ell_j \leq 8(i_j+1).
	\]
    In particular, we obtain \( \sum_{j=1}^n \Area_2(P_j) \leq 8i+8n
    \). Moreover, we have \( \Area_2(P_0) = 2(\ell_0+\ell_{\frac{1}{2}}) \)
	and \( \Area_2(P_{n+1}) = 2(\ell_{n+\frac{1}{2}} + \ell_{n+1}) \). By
	convexity, \( \ell_{\frac{1}{2}} + \ell_{n+\frac{1}{2}} \leq
	\ell_k+\ell_{n-k+1} \) holds for all \( k = 1, \dots, \lfloor \frac{n}{2}
	\rfloor\). It follows
	\[
		\Area_2(P_0) + \Area_2(P_{n+1})\ \leq \ 
		2(b_0+b_{n+1}+\ell_k+\ell_{n+k+1})\ \leq\ 
		2(b+i_k+i_{n+k+1}+2).
	\]

	By choosing \( k = 1, \dots, \lfloor \frac{n}{2}\rfloor \) such that
	\( i_k+i_{n+k+1} \) is minimal, we ensure \( i_k+i_{n+k+1} \leq
	\frac{i}{\lfloor n/2 \rfloor} \). In total, we obtain
	\[
		\Area_2(P)\ =\ \sum_{j=0}^{n+1} \Area_2(P_j)\ \leq\ 8i+8n +
		2b+\frac{2i}{\lfloor n/2 \rfloor}+4.
	\]
    For \( n = 2 \), this gives \( \Area_2(P) \leq 10i+2b+20 \), which
    strictly less than \( 12i+2b+6 \) for \( i \geq 8 \). By our
    classification~\cite{BS24}, we could verify that the claim holds
    for the finitely many polygons with \( i \leq 7 \). In particular,
    we found that the only maximizers with \( n = 2 \) are the
    finitely many polygons with \( i \in \{1,2\} \) drawn in
    figure~\ref{fig:half_integral_area_maximizers_two_interior_integral_lines}.
    For \( i = 1 \) and \( b \leq 6 \), their area even exceeds the
    generic bound, which is the reason we list \( i = 1 \) seperately
    in the assertion.

	For \( n \geq 3 \), we use the estimate \( \lfloor n/2 \rfloor \geq \frac{n-1}{2}
	\). It is then enough to show
	\[
		8i+8n+\frac{4i}{n-1}+2b+4\ <\ 12i+2b+6,
	\]
	which is equivalent to
	\[
		f(n)\ :=\ \frac{4n^2-5n+1}{2n-4}\ <\ i.
	\]
    By \cite[Proposition 2.3]{Boh23}, we get estimates \( i_j \geq j-1
    \) and \( i_{n+1-j} \geq j-1 \) for \( j = 1, \dots, \lceil
    \frac{n}{2} \rceil \). Taking the sum, we obtain \( i \geq
    \frac{n^2-2n}{4} \). For \( n \geq 11 \), this implies the claim.
    On the other hand, for \( n \leq 7 \), we have \( f(n) < 17 \).
    Again by our classification, we know the claim holds for \( i \leq
    16 \), hence this case is done. For the remaining cases \( n \in
    \{8,9,10\} \), we used our classification \cite{BS24} of maximal half-integral
    polygons with up to \( 40 \) interior lattice points to compute
    their maximally attained lattice width \( \lw_{\max} \). We
    obtained the following table:
    \begin{center}
    \begin{tabular}{c|cccccccccc}
        \( i \) & 1 & 2 & 3 & 4 & 5 & 6 & 7 & 8 & 9 & 10 \\
        \( \lw_{\max} \) & 3 & 3 & 4 & 4 & 4 & 5 & 5 & 5 &
        \nicefrac{11}{2} & 6 \\\hline
        \( i \) & 11 & 12 & 13 & 14 &15 &16 &17 &18 &19 &20 \\
        \( \lw_{\max} \) & \nicefrac{11}{2} & 6 & 6 & 6 & 7 & 7 & 7 &
        \nicefrac{15}{2} & 8 & \nicefrac{15}{2} \\\hline
        \( i \) & 21 & 22 &23 &24 &25 &26 &27 &28 &29 &30 \\
        \( \lw_{\max} \) & 8 & 8 & 8 & 8 & \nicefrac{17}{2} & 9 &
        \nicefrac{17}{2} & 9 & 9 & \nicefrac{19}{2} \\\hline
        \( i \) & 31 & 32 &33 &34 &35 &36 &37 &38 &39 &40\\
        \( \lw_{\max} \) & 10 & \nicefrac{19}{2} & \nicefrac{19}{2} &
        \nicefrac{19}{2} & 10 & 10 & 10 & 10 & 10 & 11
    \end{tabular}
    \end{center}
    We now treat \( n = 8 \). We have \( f(8) < 19 \), hence it
    suffices to show \( i \geq 19 \). Assume \( i \leq 18 \), then the
    table above implies \( n \leq \lw(P) \leq \frac{15}{2} < 8 \), a
    contradiction. The cases \( n \in \{9,10\} \) are settled
    analogously.
\end{proof}

\begin{rem}
    \label{rem:half_integral_area_maximizers}
    For half-integral polygons with at least three interior integral
    horizontal lines, the proof of
    Theorem~\ref{thm:half_integral_area_upper_bound} shows that the
    area bound is strict. This means that there are no half-integral
    area maximizers with \( n \geq 3 \). Hence the half-integral area
    maximizers are precisely the following:
    \begin{itemize}
        \item All maximizers in \( \R \times [-1,\frac{1}{2}] \),
            which are described in
            Theorem~\ref{thm:area_upper_bound},
        \item maximizers in \( \R \times [-1,1] \) as described by
            Lemma~\ref{lem:half_integral_area_bound_m1p1} and
        \item the finitely many area maximizers with \( n = 2 \) for
            \( i \in \{1,2\} \), which are drawn in
            figure~\ref{fig:half_integral_area_maximizers_two_interior_integral_lines}.
    \end{itemize}
    
\end{rem}

In the rest of this section, we show how our area bounds can be
applied to Ehrhart theory of half-integral polygons.
Recall that a rational polygon \( P
\subseteq \R^2 \) admits an \emph{Ehrhart quasipolynomial}, which
counts lattice points in its integral dilations, i.e.
\[
    \ehr_P(t)\ :=\ |tP \cap \Z^2|\ =\ \area(P)t^2+c_1(t)t+c_2(t),
    \qquad \text{for } t \in \Z_{\geq 1},
\]
where \( c_1,c_2 \colon \Z \to \Q \) are uniquely determined periodic
functions whose period divides the denominator of \( P \). For
polygons of denominator two, the four coefficients \(
c_1(1), c_1(2), c_2(1) \) and \( c_2(2) \) can be described in terms of the
numbers \( i \) and \( b \) of interior and boundary lattice points,
the area and the number \( b(2P) = |\partial P \cap \frac{1}{2}\Z^2| \) of
half-integral boundary points:
\[
    c_1(1) = \frac{b}{2}, \qquad c_2(1) = i+\frac{b}{2}-\area(P),
    \qquad c_1(2) = \frac{b(2P)}{4}, \qquad c_2(2) = 1.
\]
We refer to~\cite{BS15} for more information on Ehrhart theory in
general. The half-integral case was already studied in \cite{Her10}
with focus on the coefficients \( c_1(1) \) and \( c_2(1) \) for not
necessarily convex polygons. Moreover, the cases with period collapse,
i.e. \( c_1(1)=c_1(2) \) and \( c_2(1)=c_2(2) \), were studied in
\cite{Boh24}. Restrictions on the coefficients of Ehrhart
quasipolynomials of half-integral polygons, in particular for the case
\( i = 0 \), were also given in \cite{HHK24}.

Using our results, we contribute to the description of Ehrhart
quasipolynomials for convex half-integral polygons with at least one
interior point. Note that by the formulas for the coefficients above,
the quadruple \( (i,b,\area(P),b(2P)) \) completely determines the
Ehrhart quasipolynomial. To understand which Ehrhart quasipolynomials
are possible, we look for sharp bounds between these four parameters:
By Theorem~\ref{thm:scott_generalization}, we have the sharp bound \(
0 \leq b \leq 3i+6 \) and clearly every intermediate value is
attained. Theorems~\ref{thm:area_lower_bound}
and~\ref{thm:half_integral_area_upper_bound} give sharp lower and
upper bounds for the area in terms of \( i \) and \( b \), while
Remark~\ref{rem:area_attainers_in_between_min_max} shows that all
intermediate values are attained if \( b \geq 3 \). It remains to
bound \( b(2P) \) in terms of the first three parameters. We succeeded
in proving the following sharp lower bound:

\begin{prop}
        Let \( P \) be a half integral polygon with \( i \geq 2 \)
        interior lattice points and \( b \geq 3 \) boundary lattice
        points. Set \( A := \Area_2(P) \).
        Then the number of half-integral
        boundary points of \( P \) is bounded sharply as follows:
		\begin{displaymath}
			b(2P)\ \geq\ 2b + r +\ 
			\begin{cases}
                2, & \text{ if } A = 12i+2b+8 \text{ and } b\neq 2i+4  \\
				0, & \text{ otherwise }
			\end{cases},
		\end{displaymath}
        where \( r \) is \( 1 \) if \( A \) is odd and \( 0 \) if
        it is even.
	\end{prop}
	\begin{proof}

        Since there is always a half-integral point between two
        adjacent boundary lattice points, we get \( b(2P) \geq 2b \).
        By Pick's formula, we have
        \[
			A = \Area_1(2P)=2i(2P)+b(2P)+2.
        \]
        This implies that \( b(2P) - A \) is an even number, hence \(
        b(2P) \geq 2b + r \). We now argue that this is a sharp
        inequality for \( A \neq 12i+2b+8 \) by giving examples.
        Note that by Theorem~\ref{thm:area_lower_bound}, we have \( A
        \geq 6i+4b-6 \). Consider \( \tilde A := 12i+2b-7-A \). Then
        for \( A \neq 6i+4b-5 \), the convex hull of
        \[
			(0,1/2), (0,-1), (b-3+r/2,-1),
            \left(2i+2-\frac{1}{2}\floor{\tilde A/2},-1/2\right), (i+1,0)
        \]
        satisfies \( b(2P) = 2b+r \). For \( A = 6i+4b-5 \), the
        convex hull of
        \[
			(0,1/2), (1/2,1/2), (i+1,0), (b-3,-1), (0,-1)
        \]
        works. Moreover, for \( A = 12i+2b+8 \) and \( b = 2i+4 \), we
        can realize \( b(2P) = 2b+r \) by the following triangle:
        \[
			\conv((-1/2,-1), (2i+3/2,-1), (1/2,1)).
        \]
        It remains to show that for \( A = 12i+2b+8 \) and \( b \neq
        2i+4 \), we have the sharp inequality \( b(2P) \geq 2b + 2 \)
        (note that \( r = 0 \) in this case). By Remark
        \ref{rem:half_integral_area_maximizers}, there are three kinds
        of half-integral area maximizers: Those in \( \R \times
        [-1,\frac{1}{2}] \) are described by
        Theorem~\ref{thm:area_upper_bound} and we can check that \(
        b(2P) = 2b+2 \) holds for all of them. The same is true for
        the finitely many area maximizers with two interior integral
        lines and \( i = 2 \) drawn in
        figure~\ref{fig:half_integral_area_maximizers_two_interior_integral_lines}.
        The remaining ones are described by
        lemma~\ref{lem:half_integral_area_bound_m1p1}. Here, \( b \neq
        2i+4 \) implies that we have at least one vertex with second
        coordinate \( \pm \frac{1}{2} \). Hence there is at least one
        pair of adjacent non-intgral vertices with no lattice point
        betwen them, which implies \( b(2P) \geq 2b+2 \).
	\end{proof}

    It remains to bound \( b(2P) \) from above in terms of \( i, b \)
    and \( \area(P) \) and to understand which intermediate values are
    attained. Together with our other bounds, this would allow the
    complete description of the four-parameter family \(
    (i(P),b(P),\area(P),b(2P)) \) and hence all possible Ehrhart
    quasipolynomials of half-integral polygons having \( i \geq 1 \).
    In particular, this should provide a path to prove that the number
    of distinct Ehrhart quasipolynomial of half-integral polygons is
    given by the following polynomial, which we have already observed
    empirically from our classification:

    \begin{conj}[See also~{\cite[Conjecture 6.4]{BS24}}]
	There are exactly 
    \[
	\frac{9}{2}i^3+36i^2+\frac{175}{2}i+53
    \]
    Ehrhart quasi-polynomials of half-integral polygons with $i\in
    \Z_{\geq 2}$ interior lattice points and at least \( 3 \) boundary lattice
    points.
	\end{conj}

    \printbibliography

    \appendix
    \listoffigures

\begin{figure}[H]
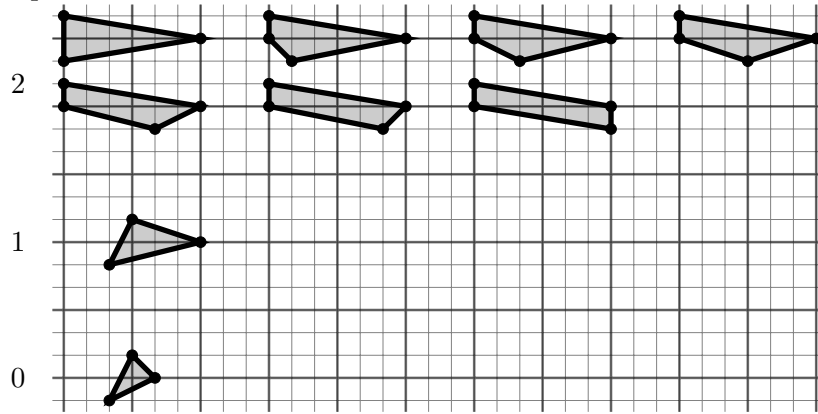


    \caption[Area minimizers with collinear lattice points from
    Theorem~\ref{thm:area_lower_bound}]{All area minimizers with collinear lattice points as described by
		Theorem~\ref{thm:area_lower_bound} for \( (k,i) = (3,1) \). On
        the left is the number \( b \) of boundary lattice points.}
    \label{fig:area_minimizers_collinear_integer_hull}


\end{figure}
\begin{figure}[H]
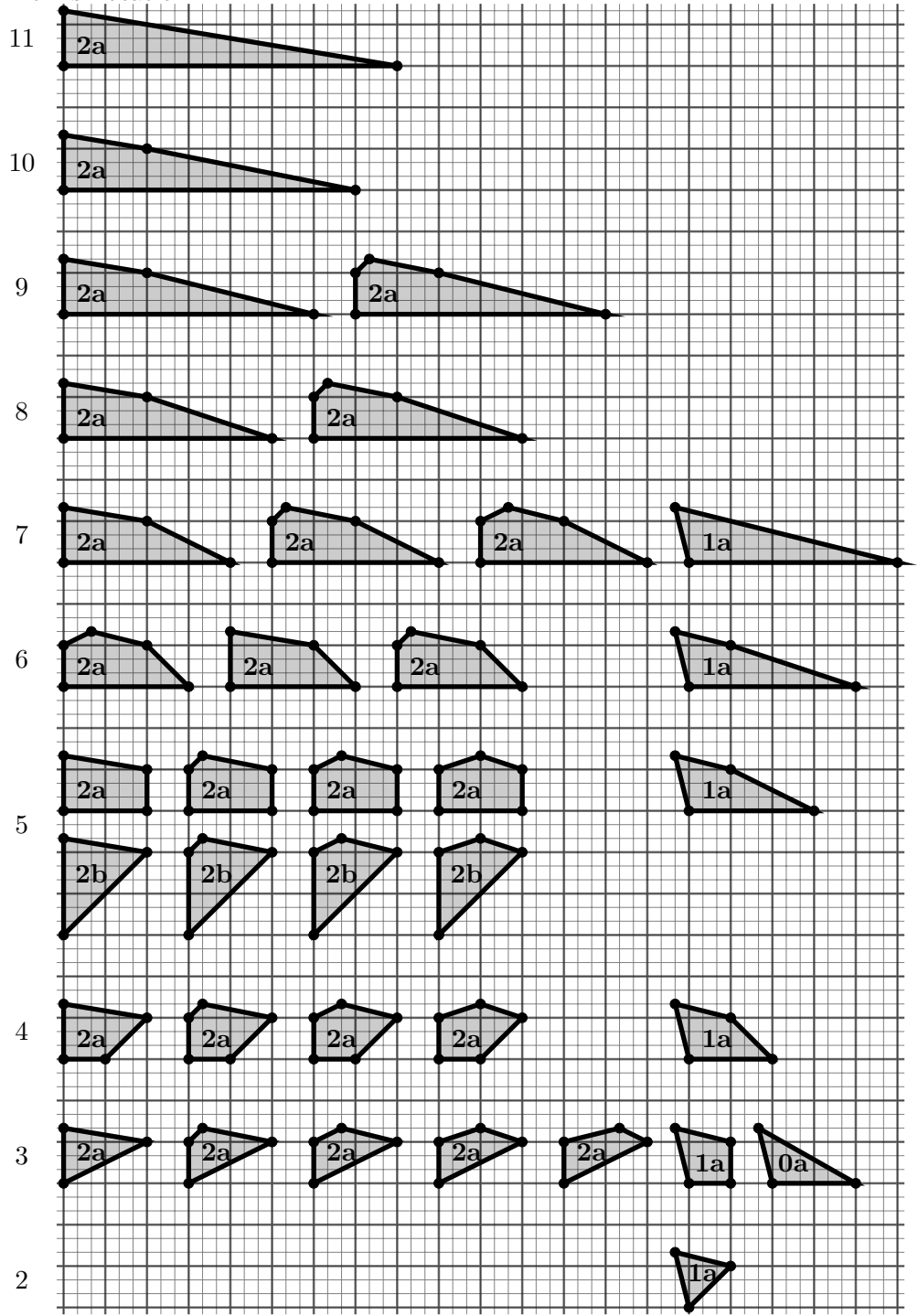

    \caption[Area minimizers with two-dimensional integer hull from
    Theorem~\ref{thm:area_lower_bound}]{All area minimizers with two-dimensional integer hull as described by
		Theorem~\ref{thm:area_lower_bound} for \( (k,i) = (3,1) \). On
    the left is the number \( b \) of boundary lattice points. Each
polygon is labeled according to the Theorem's notation.}
\label{fig:area_minimizers}

%
\end{figure}

\begin{figure}[H]
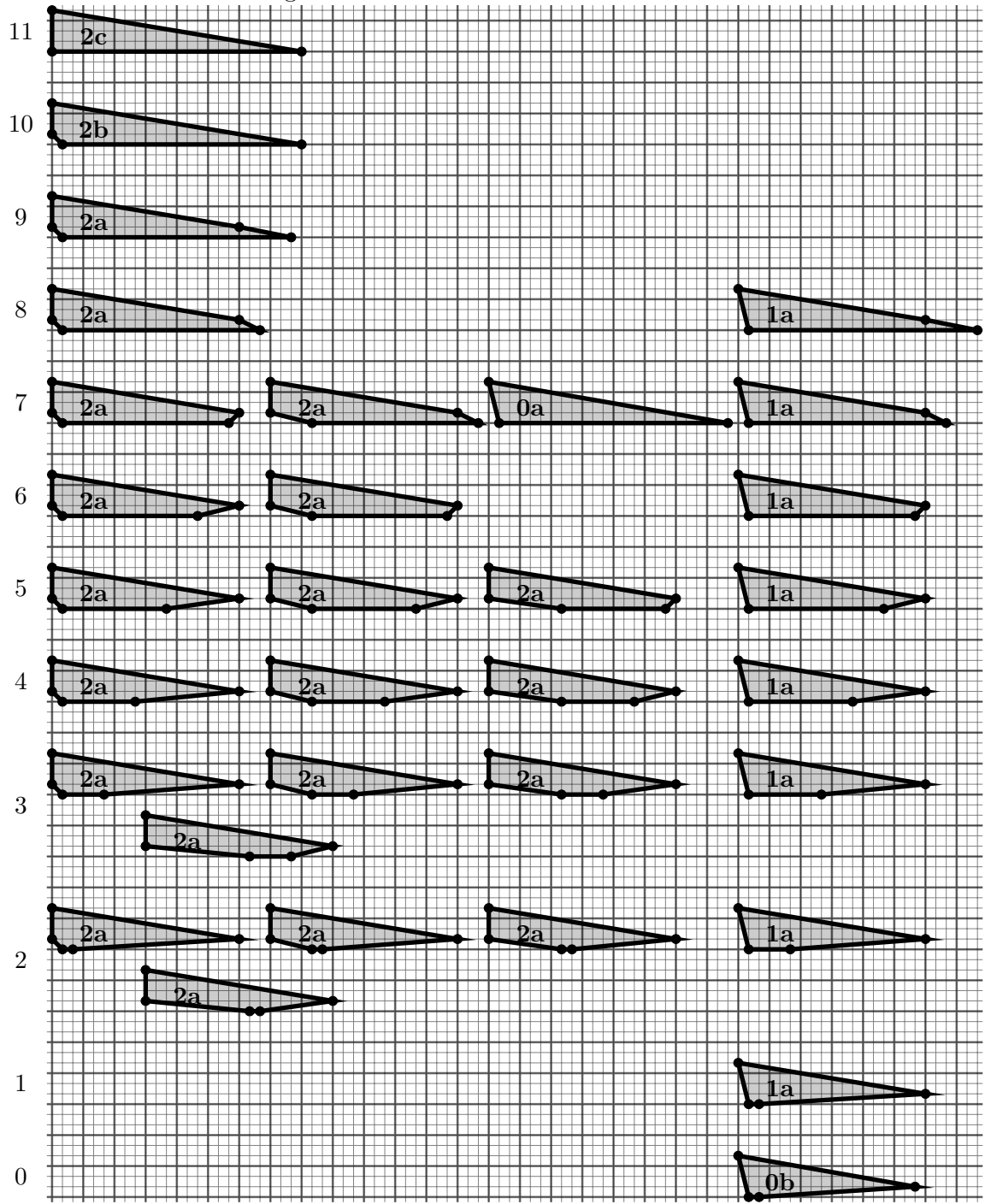

    \caption[Area maximizers from Theorem~\ref{thm:area_upper_bound}.]{All area maximizers as described by
		Theorem~\ref{thm:area_upper_bound} for \( (k,i) = (3,1) \). On
        the left is the number \( b \) of boundary lattice points. Each
		polygon is labeled according to the Theorem's notation.}
    \label{fig:area_maximizers}
	
%
\end{figure}

\begin{figure}[H]
    \caption[Half-integral area maximizers with two interior integral
lines]{All half-integral area maximizers with two interior
        integral lines. On the left is the number \( b \) of boundary
        lattice points. The dashed line seperates polygons with \( i = 1 \)
    from polygons with \( i = 2 \).}
    \label{fig:half_integral_area_maximizers_two_interior_integral_lines}
%
\end{figure}

\begin{figure}[H]
    \caption[Half-integral area maximizers with one interior integral
    line as described by
    Lemma~\ref{lem:half_integral_area_bound_m1p1}]{All half-integral area maximizers in \( \R \times [-1,1]
        \) with two interior lattice points and five boundary lattice
        points, which cannot be realized in \( \R \times [-1,\frac{1}{2}]
        \), as described by
    Lemma~\ref{lem:half_integral_area_bound_m1p1}. The upper two rows
of polygons have \( b_0 = 2 \), the polygons in the lower two rows
have \( b_0 = 1 \)}
    \label{fig:half_integral_area_maximizers_m1p1}
%
\end{figure}

\begin{figure}[H]
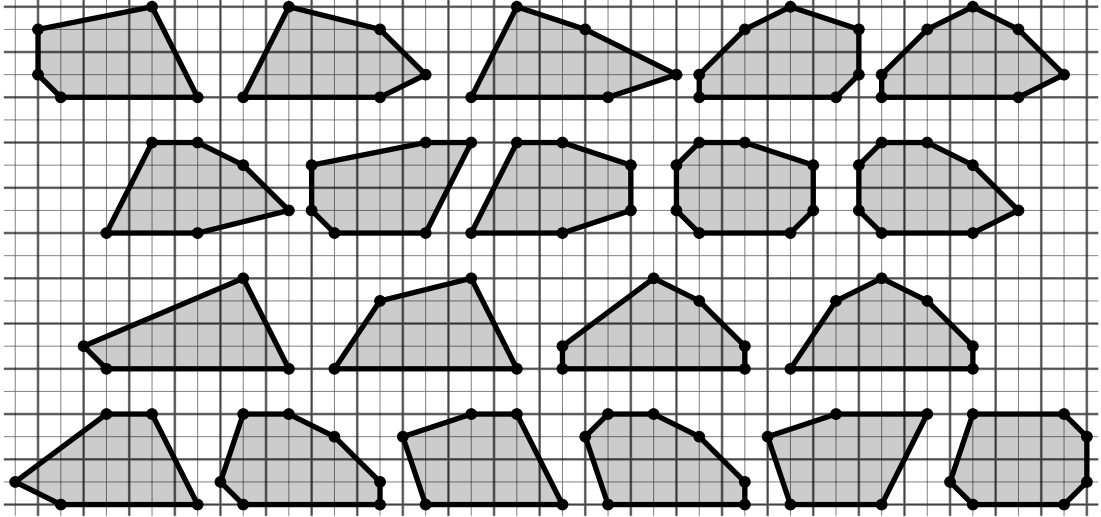

	\caption[Polygons for areas in between the lower and upper bound
    according to Remark~\ref{rem:area_attainers_in_between_min_max}]
    {A polygon for each possible area between the lower and
    upper bounds of Theorems~\ref{thm:area_lower_bound}
    and~\ref{thm:area_upper_bound}. Here, we have \( (k,i,b) =
(3,1,8) \). Each polygon is labeled according to its \( k \)-normalized
area and the steps between them are labeled according to
Remark~\ref{rem:area_attainers_in_between_min_max}.}
	\label{fig:area_attainers_in_between_min_max}

%

\end{figure}

\end{document}